\documentclass[12pt]{article}

\usepackage{amsmath,amssymb,amsbsy,amsfonts,amsthm,latexsym,
            amsopn,amstext,amsxtra,amscd,stmaryrd}
\usepackage{fullpage}
\usepackage[mathscr]{eucal}

\begin{document}

\newtheorem{theorem}{Theorem}
\newtheorem{lemma}[theorem]{Lemma}
\newtheorem{corollary}[theorem]{Corollary}
\newtheorem{proposition}[theorem]{Proposition}
\newtheorem{conjecture}[theorem]{Conjecture}

\theoremstyle{definition}
\newtheorem*{definition}{Definition}
\newtheorem*{remark}{Remark}
\newtheorem*{example}{Example}

% CALLIGRAPHIC ALPHABET

\def\cA{\mathcal A}
\def\cB{\mathcal B}
\def\cC{\mathcal C}
\def\cD{\mathcal D}
\def\cE{\mathcal E}
\def\cF{\mathcal F}
\def\cG{\mathcal G}
\def\cH{\mathcal H}
\def\cI{\mathcal I}
\def\cJ{\mathcal J}
\def\cK{\mathcal K}
\def\cL{\mathcal L}
\def\cM{\mathcal M}
\def\cN{\mathcal N}
\def\cO{\mathcal O}
\def\cP{\mathcal P}
\def\cQ{\mathcal Q}
\def\cR{\mathcal R}
\def\cS{\mathcal S}
\def\cU{\mathcal U}
\def\cT{\mathcal T}
\def\cV{\mathcal V}
\def\cW{\mathcal W}
\def\cX{\mathcal X}
\def\cY{\mathcal Y}
\def\cZ{\mathcal Z}

% SCRIPT ALPHABET

\def\sA{\mathscr A}
\def\sB{\mathscr B}
\def\sC{\mathscr C}
\def\sD{\mathscr D}
\def\sE{\mathscr E}
\def\sF{\mathscr F}
\def\sG{\mathscr G}
\def\sH{\mathscr H}
\def\sI{\mathscr I}
\def\sJ{\mathscr J}
\def\sK{\mathscr K}
\def\sL{\mathscr L}
\def\sM{\mathscr M}
\def\sN{\mathscr N}
\def\sO{\mathscr O}
\def\sP{\mathscr P}
\def\sQ{\mathscr Q}
\def\sR{\mathscr R}
\def\sS{\mathscr S}
\def\sU{\mathscr U}
\def\sT{\mathscr T}
\def\sV{\mathscr V}
\def\sW{\mathscr W}
\def\sX{\mathscr X}
\def\sY{\mathscr Y}
\def\sZ{\mathscr Z}

% FRAKTUR ALPHABET

\def\fA{\mathfrak A}
\def\fB{\mathfrak B}
\def\fC{\mathfrak C}
\def\fD{\mathfrak D}
\def\fE{\mathfrak E}
\def\fF{\mathfrak F}
\def\fG{\mathfrak G}
\def\fH{\mathfrak H}
\def\fI{\mathfrak I}
\def\fJ{\mathfrak J}
\def\fK{\mathfrak K}
\def\fL{\mathfrak L}
\def\fM{\mathfrak M}
\def\fN{\mathfrak N}
\def\fO{\mathfrak O}
\def\fP{\mathfrak P}
\def\fQ{\mathfrak Q}
\def\fR{\mathfrak R}
\def\fS{\mathfrak S}
\def\fU{\mathfrak U}
\def\fT{\mathfrak T}
\def\fV{\mathfrak V}
\def\fW{\mathfrak W}
\def\fX{\mathfrak X}
\def\fY{\mathfrak Y}
\def\fZ{\mathfrak Z}

% BLACKBOARD BOLD

\def\C{{\mathbb C}}
\def\F{{\mathbb F}}
\def\K{{\mathbb K}}
\def\L{{\mathbb L}}
\def\N{{\mathbb N}}
\def\P{{\mathbb P}}
\def\Q{{\mathbb Q}}
\def\R{{\mathbb R}}
\def\Z{{\mathbb Z}}

% SOME STANDARD DEFINITIONS

\def\eps{\varepsilon}
\def\mand{\quad\text{and}\quad}
\def\\{\cr}
\def\({\left(}
\def\){\right)}
\def\[{\left[}
\def\]{\right]}
\def\<{\langle}
\def\>{\rangle}
\def\fl#1{\left\lfloor#1\right\rfloor}
\def\rf#1{\left\lceil#1\right\rceil}
\def\le{\leqslant}
\def\ge{\geqslant}
\def\ds{\displaystyle}
\def\Prim{\text{\tt Prim}}

\def\xxx{\vskip5pt\hrule\vskip5pt}
\def\imhere{ \xxx\centerline{\sc I'm here}\xxx }

\newcommand{\comm}[1]{\marginpar{
\vskip-\baselineskip \raggedright\footnotesize
\itshape\hrule\smallskip#1\par\smallskip\hrule}}

% SPECIAL DEFINITIONS FOR THIS PAPER

\def\sDge#1{\sD_{\raisebox{1pt}{$\scriptscriptstyle\ge$}#1}}
\def\Nge#1{\N_{\raisebox{1pt}{$\scriptscriptstyle\ge$}#1}}

%%%%%%%%%%%%%%%%%%%%%%%%%%%%%%%%%%%%%%%%%
%%%%%%%%%%  PAPER STARTS HERE  %%%%%%%%%%
%%%%%%%%%%%%%%%%%%%%%%%%%%%%%%%%%%%%%%%%%

\title{\bf Optimal primitive sets with restricted primes}

\author{{\sc William D.~Banks} \\
{Department of Mathematics} \\
{University of Missouri} \\
{Columbia, MO 65211 USA} \\
{\tt bankswd@missouri.edu} 
\and
{\sc Greg Martin} \\
{Department of Mathematics} \\
{University of British Columbia} \\
{Room 121, 1984 Mathematics Road} \\
{Vancouver, V6T 1Z2 Canada} \\
{\tt gerg@math.ubc.ca}}

\date{\today}
\pagenumbering{arabic}

\maketitle

\begin{abstract}
A set of natural numbers is {\em primitive} if no element of the set
divides another. Erd\H os conjectured that if $\sS$ is any primitive set, then
$$
\sum_{n\in\sS}\frac{1}{n\log n} \le \sum_{p\in\P}\frac{1}{p\log p},
$$
where $\P$ denotes the set of primes. In this paper, we make progress towards this conjecture by restricting
the setting to smaller sets of primes. Let $\sP$ denote any subset of $\P$, and let $\N(\sP)$ denote the
set of
natural numbers all of whose prime factors are in~$\sP$.
We say that $\sP$ is {\em Erd\H os-best among primitive subsets of $\N(\sP)$} if the inequality
$$
\sum_{n\in\sS}\frac{1}{n\log n} \le \sum_{p\in\sP}\frac{1}{p\log p}
$$
holds for every primitive set $\sS$ contained in $\N(\sP)$. We show that if the sum of the reciprocals of the elements of $\sP$ is small enough, then $\sP$ is Erd\H os-best among primitive subsets of $\N(\sP)$. As an application, we prove that the set of twin primes exceeding $3$ is Erd\H os-best among the corresponding primitive sets.

This problem turns out to be related to a similar problem involving multiplicative weights. For any real number $t>1$, we say that $\sP$ is {\em $t$-best among primitive subsets of $\N(\sP)$} if the inequality
$$
\sum_{n\in\sS} n^{-t} \le \sum_{p\in\sP} p^{-t}
$$
holds for every primitive set $\sS$ contained in $\N(\sP)$. We show that if
the sum on
the right-hand side of this inequality is small enough,
then $\sP$ is $t$-best among primitive subsets of $\N(\sP)$.
\end{abstract}

\newpage

\section{Introduction}
\label{sec:intro}

A nonempty set of natural numbers is called \emph{primitive} if no element
of the set divides another (for later convenience, we stipulate that the singleton
set $\{1\}$ is not primitive). In 1935, Erd\H os~\cite{Erdos} established the convergence
of the sum of $1/(n\log n)$ over all elements $n$ of a given primitive set; from this he
deduced that the lower asymptotic density of a primitive set must equal~$0$ (in contrast to the
upper density, which can be positive, as shown by Besikovitch~\cite{Bes}).
Erd\H os actually proved that this sum of $1/(n\log n)$ is bounded by a universal constant:
\begin{equation*}
\sup\limits_{\sS \text{ primitive}} \sum_{n\in\sS}\frac{1}{n\log n}<\infty.
\end{equation*}
Noting that the set $\P$ of all primes is itself primitive and contains many small elements,
Erd\H os proposed that the supremum on the left-hand side is attained
when $\sS=\P$.

\begin{conjecture}[Erd\H os] \label{primes are E-best conj}
For any primitive set $\sS$, we have
$\displaystyle \sum_{n\in\sS} \frac1{n\log n} \le \sum_{p\in\P} \frac1{p\log p}\,$.
\end{conjecture}

\noindent This conjecture is still open, although it has been established for primitive sets $\sS$ with additional properties (see for example~\cite{Zhang}).

We consider a generalization of this problem, to primitive sets whose elements are restricted to have only certain prime factors. For a given set of primes $\sP$, let $\N(\sP)$ denote the set of natural numbers
divisible only by primes in $\sP$ (the multiplicative semigroup generated by $\sP$), that is,
$$
\N(\sP) = \{n\in\N\colon p\mid n\Rightarrow p\in\sP\}.
$$
We say that $\sP$ is {\em Erd\H os-best among primitive subsets of $\N(\sP)$} if the inequality
$$
\sum_{n\in\sS}\frac{1}{n\log n} \le \sum_{p\in\sP}\frac{1}{p\log p}
$$
holds for every primitive set $\sS$ contained in $\N(\sP)$. In this terminology,
Conjecture~\ref{primes are E-best conj} can be restated as the assertion that
$\P$ is Erd\H os-best among primitive subsets of~$\N$. A similar heuristic,
together with some computational evidence, leads us to generalize
the conjecture of Erd\H os
to these restricted sets.

\begin{conjecture} \label{restricted primes are E-best conj}
Any set of primes $\sP$ is Erd\H os-best among primitive subsets of $\N(\sP)$.
\end{conjecture}

Our first result shows that this conjecture holds if the sum of the reciprocals of the elements of $\sP$ is small enough.

\begin{theorem}
\label{thm:main}
Let $\sP$ be a set of primes such that
\begin{equation}
\label{eq:antarctica}
\sum_{p\in\sP}p^{-1}\le 1+\bigg(1-\sum_{p\in\sP}p^{-2}\bigg)^{1/2},
\end{equation}
Then $\sP$ is Erd\H os-best among primitive subsets of $\N(\sP)$.
\end{theorem}

\begin{remark}
The square root on the right-hand side of the inequality~\eqref{eq:antarctica} is always well-defined, as $\sum_{p\in\sP}p^{-2} \le \sum_{p\in\P}p^{-2} < 1$. In fact, the latter sum can be precisely evaluated using the rapidly converging series
\begin{equation} \label{less than 1}
\sum_{p\in\P}p^{-2} = \sum_{m=1}^\infty \frac{\mu(m)}{m}\log\zeta(2m) = 0.45224742\cdots.
\end{equation}
We conclude that if $\sP$ satisfies the inequality~\eqref{eq:antarctica} then $\sum_{p\in\sP}p^{-1} < 2$, while if $\sum_{p\in\sP}p^{-1} \le 1+\big(1-\sum_{p\in\P}p^{-2}\big)^{1/2} = 1.74010308\cdots$ then $\sP$ satisfies the inequality~\eqref{eq:antarctica}.
\end{remark}

The following application of Theorem~\ref{thm:main} is quickly derived
in Section~\ref{sec:twinprimes}.

\begin{corollary} \label{cor:twin primes}
Let $\sT$ denote the set of twin primes exceeding $3$, that is, the set of all primes $p>3$ for which $p-2$ or $p+2$ is also prime. Then $\sT$ is Erd\H os-best among primitive subsets of $\N(\sT)$.
\end{corollary}

\noindent We find it amusing that we can identify the optimal primitive subset of $\N(\sT)$ without needing to determine whether that subset is finite or infinite!

We can further generalize this problem by demanding that the integers in our sets have at least a certain number of prime factors. For every natural number $k$, define
$$
\N_k = \{ n\in \N\colon \Omega(n)=k \} \mand \Nge k = \{ n\in \N\colon \Omega(n)\ge k \},
$$
where as usual $\Omega(n)$ denotes the number of prime factors of $n$ counted with multiplicity; for example, $\N_1 = \P$ and $\Nge1 = \N\setminus\{1\}$. Note that each of the sets $\N_k$ is itself a primitive set. One step towards a proof of Conjecture~\ref{primes are E-best conj} would thus be to establish the natural conjecture
\begin{equation*}
\sum_{p\in\P} \frac1{p\log p} > \sum_{n\in\N_2} \frac1{n\log n} > \sum_{n\in\N_3} \frac1{n\log n} > \cdots,
\end{equation*}
but this is still an open problem: it was shown by Zhang~\cite{Zhang} that the first sum over $\P$ is indeed larger than any of the other sums, but even this partial result is nontrivial.
However, in the setting of restricted prime factors, we can establish the analogous
chain of inequalities and in fact more. For any set $\sP$ of primes, define
$$
\N_k(\sP) = \N(\sP) \cap \N_k \mand \Nge k(\sP) = \N(\sP) \cap \Nge k,
$$
so that we are now simultaneously restricting the allowable prime factors $\sP$ and the minimum number of prime factors~$k$.

\begin{theorem}
\label{thm:main2}
Let $\sP$ be a set of primes for which the inequality~\eqref{eq:antarctica} holds. Then for every natural number $k$, the set $\N_k(\sP)$ is Erd\H os-best among primitive subsets of $\Nge{k}(\sP)$.
\end{theorem}

\noindent Manifestly, Theorem~\ref{thm:main} is the special case $k=1$ of Theorem~\ref{thm:main2}. Since $\N_k(\sP) \subset \Nge j(\sP)$ for all $k\ge j$, Theorem~\ref{thm:main2} implies in particular that when $\sP$ is a set of primes satisfying the inequality~\eqref{eq:antarctica},
we have
\begin{equation} \label{chain of inequalities}
\sum_{p\in\sP} \frac1{p\log p} > \sum_{n\in\N_2(\sP)} \frac1{n\log n} > \sum_{n\in\N_3(\sP)} \frac1{n\log n} > \cdots.
\end{equation}
We do not formulate Theorem~\ref{thm:main2} simply for its own sake: our proof of Theorem~\ref{thm:main} requires comparing sets containing elements with different numbers of prime factors, and the chain of inequalities~\eqref{chain of inequalities} is a stable yardstick upon which these comparisons can be made.

Finally, we modify the problem in yet a different way. Instead of counting an integer $n$ with weight $1/(n\log n)$ in these sums, we may instead count it with weight $n^{-t}$ for some fixed real number~$t$. We can establish an analogue of Theorem~\ref{thm:main2} for these weights as well. We say that $\N_k(\sP)$ is {\em $t$-best among primitive subsets of $\Nge k(\sP)$} if the inequality
$$
\sum_{n\in\sS} n^{-t} \le \sum_{p\in\N_k(\sP)} n^{-t}
$$
holds for every primitive set $\sS$ contained in $\N_k(\sP)$.

\begin{theorem}
\label{thm:t-best}
Let $t>1$ be a real number, and let $\sP$ be a set of primes satisfying the inequality
\begin{equation}
\label{eq:polar-region}
\sum_{p\in\sP}p^{-t}\le 1+\bigg(1-\sum_{p\in\sP}p^{-2t}\bigg)^{1/2}.
\end{equation}
Then for every natural number $k$, the set $\N_k(\sP)$ is $t$-best among primitive subsets of $\Nge{k}(\sP)$.
\end{theorem}

\noindent In fact, it suffices to establish Theorem~\ref{thm:t-best} when $\sP$ is finite, as we show in Section~\ref{sec:twinprimes}.

It can be verified that the function
$$
\sum_{p\in\P}p^{-t} - 1 - \bigg(1-\sum_{p\in\P}p^{-2t}\bigg)^{1/2} = \sum_{m=1}^\infty \frac{\mu(m)}m \log \zeta(tm) - 1 - \bigg(1-\sum_{m=1}^\infty \frac{\mu(m)}m \log \zeta(2tm)\bigg)^{1/2},
$$
defined for $t>1$, has a unique  zero $\tau=1.1403659\cdots$ and is positive for $t>\tau$. Furthermore, by monotonicity, if the inequality~\eqref{eq:polar-region} is satisfied for $\sP=\P$ then it is satisfied for any set of primes. We therefore have the following corollary.

\begin{corollary}
If $t>\tau$, then any set of primes $\sP$ is $t$-best among all primitive subsets of~$\N(\sP)$.
\end{corollary}

\noindent This assertion does not necessarily hold for every $t$: in fact, if $\sum_{p\in\sP} p^{-1}$ diverges, then one can establish, using equation~\eqref{N2big} below, the existence of a number $\delta(\sP)>0$ such that $\sum_{p\in\sP} p^{-t} < \sum_{n\in\N_2(\sP)} n^{-t}$ for all $t$ between $1$ and $1+\delta(\sP)$.

These new weighted sums $\sum_{n\in\sS} n^{-t}$ are much easier to handle than the original sums $\sum_{n\in\sS} 1/(n\log n)$, because $n^{-t}$ is a multiplicative function of~$n$. However, Theorem~\ref{thm:t-best} is not merely analogous to Theorem~\ref{thm:main2}: in the next section we actually derive the latter from the former. Once Section~\ref{sec:twinprimes} is done, the only remaining task is to prove Theorem~\ref{thm:t-best}, which we accomplish in Section~\ref{sec:proof}.

\section{Quick derivations}
\label{sec:twinprimes}

In this section we give the three quick derivations described in the introduction: first we demonstrate that the infinite case of Theorem~\ref{thm:main} follows from the finite case, then we deduce Theorem~\ref{thm:main2} (of which Theorem~\ref{thm:main} is a special case) from Theorem~\ref{thm:t-best}, and finally we derive Corollary~\ref{cor:twin primes} from Theorem~\ref{thm:main}. For convenience we introduce the notation
$$
\Sigma_t(\sS) = \sum_{n\in\sS} n^{-t}
$$
for any set $\sS\subset\N$ and any $t>1$.

\begin{proof}[Proof of Theorem~\ref{thm:main} for infinite $\sP$,
assuming Theorem~\ref{thm:main} for finite $\sP$]
Let $t>1$ be a real number, let $\sP$ be an infinite set of primes
satisfying the inequality~\eqref{eq:polar-region}, and let $\sS$ be any primitive subset of $\N(\sP)$; we want to show that $\Sigma_t(\sS) \le \Sigma_t(\sP)$. Enumerate $\sP$ as $\sP = \{p_1,p_2,\dots\}$, and for each natural number $n$, let $\sP_n = \{p_1,\dots,p_n\}$. Also let $\sS_n = \sS \cap \N(\sP_n)$, so that $\sS_n$ is a primitive subset of $\N(\sP_n)$. Note that the $\sP_n$ form a nested sequence of sets whose union is $\sP$, and similarly for $\sS_n$ and~$\sS$.

Because the inequality~\eqref{eq:polar-region} holds for $\sP$, it also holds for each $\sP_n$ by monotonicity:
$$
\Sigma_t(\sP_n) < \Sigma_t(\sP) \le 1 + \big( 1-\Sigma_{2t}(\sP) \big)^{1/2} < 1 + \big( 1-\Sigma_{2t}(\sP_n) \big)^{1/2}.
$$
Therefore we may apply Theorem~\ref{thm:main} to the finite set $\sP_n$ for each natural number $n$, concluding that $\Sigma_t(\sS_n) \le \Sigma_t(\sP_n)$. Taking the limit as $n$ tends to infinity (valid by the dominated convergence theorem, for example), we deduce that $\Sigma_t(\sS) \le \Sigma_t(\sP)$ as desired.
\end{proof}

\begin{lemma}
\label{lem:where}
Let $\sU$ be a subset of $\N$, and let $\sS^\star$ be a primitive subset of~$\sU$. If $\sS^\star$ is $t$-best among all primitive subsets of $\sU$ for every $t>1$, then $\sS^\star$ is also Erd\H os-best among all primitive subsets of~$\sU$.
\end{lemma}

\begin{proof}
If $\sS$ is any primitive subset of $\sU$, then
$$
\sum_{n\in\sS}\frac{1}{n\log n}
=\sum_{n\in\sS}\int_1^\infty n^{-t}\,dt
=\int_1^\infty\sum_{n\in\sS}n^{-t}\,dt
=\int_1^\infty \Sigma_t(\sS)\,dt.
$$
(The leftmost sum is finite by Erd\H os's result~\cite{Erdos},
and the interchange of integral and sum is justified because all terms are positive.)
The hypothesis that $\sS^\star$ is $t$-best among all primitive subsets of $\sU$ for
every $t>1$ means that $\Sigma_t(\sS)\le\Sigma_t(\sS^\star)$ for every $t>1$. It follows that
$$
\sum_{n\in\sS} \frac{1}{n\log n}=\int_1^\infty \Sigma_t(\sS)\,dt
\le\int_1^\infty \Sigma_t(\sS^\star)\,dt= \sum_{n\in\sS^\star} \frac{1}{n\log n}\,,
$$
as required.
\end{proof}

\begin{proof}[Proof of Theorem~\ref{thm:main2}, assuming Theorem~\ref{thm:t-best}]
If $\sP$ is a set of primes for which the inequality~\eqref{eq:antarctica} holds, then for any $t>1$,
$$
\sum_{p\in\sP}p^{-t} < \sum_{p\in\sP}p^{-1} \le 1+\bigg(1-\sum_{p\in\sP}p^{-2}\bigg)^{1/2} < 1+\bigg(1-\sum_{p\in\sP}p^{-2t}\bigg)^{1/2}.
$$
Theorem~\ref{thm:t-best} implies that $\N_k(\sP)$ is $t$-best among primitive subsets of $\Nge{k}(\sP)$ for every natural number~$k$ and every $t>1$. It follows from Lemma~\ref{lem:where} that $\N_k(\sP)$ is also Erd\H os-best among primitive subsets of~$\Nge{k}(\sP)$.
\end{proof}

\begin{proof}[Proof of Corollary~\ref{cor:twin primes}, assuming Theorem~\ref{thm:main}]
It suffices to verify that the inequality~\eqref{eq:antarctica} is satisfied with the set $\sT=\{5,7,11,13,\ldots\}$
consisting of twin primes exceeding~$3$. On one hand, if $B$ is the Brun constant defined by
$$
B=\sum_{p\colon p+2\in\P}\bigg(\frac{1}{p}+\frac{1}{p+2}\bigg) = \bigg( \frac13 + \frac15 \bigg) + \bigg( \frac15 + \frac17 \bigg) + \bigg( \frac1{11} + \frac1{13} \bigg) + \cdots,
$$
then the bound $B<2.347$ has been given by Crandall and
Pomerance~\cite[pp.~16--17]{CranPom} (for a proof, see Klyve~\cite[Chapter~3]{Klyve}), and therefore
\begin{equation}
\label{eq:diving}
\sum_{p\in\sT}p^{-1}=B-\frac13-\frac15<1.814.
\end{equation}
On the other hand, we have
\begin{align}
\sum_{p\in\sT}p^{-2} &< \sum_{n=1}^\infty \bigg( \frac1{(6n-1)^2} + \frac1{(6n+1)^2} \bigg) \notag \\
&< \sum_{n=1}^\infty \bigg( \frac1{6n(6n-3)} + \frac1{6n(6n+3)} \bigg) \notag \\
&= \sum_{n=1}^\infty \frac19 \bigg( \frac1{2n-1} - \frac1{2n+1} \bigg) = \frac19\,.
\label{eq:board}
\end{align}
We conclude that
$$
\sum_{p\in\sT}p^{-1} < 1.814 < 1.9428 < 1+\bigg(1-\frac19\bigg)^{1/2} < 1+\bigg(1-\sum_{p\in\sT}p^{-2}\bigg)^{1/2},
$$
and
thus Theorem~\ref{thm:main} can be applied to deduce that
$\sT$ is Erd\H os-best among primitive subsets of $\N(\sT)$.
\end{proof}

Let $\sT_3 = \sT \cup \{3\}$ be the set of all twin primes, including~$3$. It can be shown that $\sum_{p\in\sT_3}p^{-2}$ is between $0.19725177$ and $0.19725181$. To show that $\sT_3$ is Erd\H os-best among all primitive subsets of $\N(\sT_3)$, it therefore suffices to establish the unconditional bound $B<2.0959621$ on the Brun constant. The true value of Brun's constant is believed to be $B=1.90216\cdots$ (see for example Sebah and Demichel~\cite{SebDem}),
and if this is the case, then it follows from Theorem~\ref{thm:main} that $\sT_3$ is indeed Erd\H os-best.
Regrettably, the value $B=1.90216\cdots$ is quoted in several places in the literature in a manner
that suggests it has been rigorously established, but at the present time no bound better than
$B<2.347$ is known unconditionally.

\section{Proof of Theorem~\ref{thm:t-best}}
\label{sec:proof}

We now turn to the
sole remaining task, namely, establishing Theorem~\ref{thm:t-best} when the set of primes $\sP$ is finite;
we accomplish this task with the more detailed Proposition~\ref{prop:what} stated below.
As before, $\N(\sP)$ denotes the set of natural numbers all of whose prime factors lie in~$\sP$.
We recall the previously defined notation
$$
\N_k=\{n\in\N\colon \Omega(n)=k\}\mand
\Nge{k}=\{n\in\N\colon \Omega(n)\ge k\}
$$
(with $\N_0 = \{1\}$), as well as
$$
\N_k(\sP)=\N(\sP)\cap\N_k\mand
\Nge{k}(\sP)=\N(\sP)\cap\Nge{k}.
$$
We also recall the notation $\Sigma_t(\sS)=\sum_{n\in\sS}n^{-t}$ for any set $\sS$ of natural numbers.

\begin{lemma}
\label{lem:who}
Let $t>1$ be a real number, and let $\sP$ be a finite set of primes. Suppose that
\begin{itemize}
\item[$(i)$] for every proper subset $\sQ$ of $\sP$ and for every $k\in\N$, the primitive set $\N_k(\sQ)$ is $t$-best among all primitive subsets of $\Nge{k}(\sQ)$;
\item[$(ii)$] the inequality $\Sigma_t(\N_k(\sP))\ge\Sigma_t(\N_{k+1}(\sP))$ holds for all $k\in\N$.
\end{itemize}
Then $\N_k(\sP)$ itself is $t$-best among all primitive subsets of $\Nge{k}(\sP)$, for every $k\in\N$.
\end{lemma}

\begin{proof}
Fix $k\in\N$, and let $\sS$ be a primitive subset of $\Nge{k}(\sP)$; we need to show that $\Sigma_t(\sS) \le \Sigma_t(\N_k(\sP))$. Define
$$
\ell = \min\{\Omega(n)\colon n\in\sS\}
$$
(so that $\ell\ge k$), and fix a number $s\in\sS$ with $\Omega(s)=\ell$. We proceed to partition both $\N_\ell(\sP)$ and $\sS$ according to the greatest common divisor of their elements with~$s$.

Let $d$ denote any divisor of $s$. Notice that
$$
\{ n\in \N_\ell(\sP)\colon (n,s) = d \} = d \cdot \big\{ m\in \tfrac1d\N_\ell(\sP) \cap\N \colon \big(m,\tfrac sd\big) = 1 \big\} = d \cdot \N_{\ell-\Omega(d)}(\sQ_d),
$$
where $\sQ_d$ is the set of primes in $\sP$ that do not divide $s/d$; note that $\sQ_d$ is a proper subset of $\sP$ when $d\ne s$. We define
$$
\sS_d = \tfrac1d \{ n\in \sS \colon (n,s) = d \} = \big\{ m\in \tfrac1d\sS \cap\N \colon \big(m,\tfrac sd\big) = 1 \big\},
$$
noting that $\sS_d$ is a primitive subset of $\Nge{\ell-\Omega(d)}(\sQ_d)$.
With this notation, the sets $\N_\ell(\sP)$ and $\sS$ can be decomposed as the disjoint unions
$$
\N_\ell(\sP)=\bigcup\limits_{d\mid s} d\cdot \N_{\ell-\Omega(d)}(\sQ_d) \mand
\sS=\bigcup\limits_{d\mid s} d\cdot\sS_d.
$$
Therefore
\begin{align}
\label{eq:lunch2}
\Sigma_t(\N_\ell(\sP))-\Sigma_t(\sS)
&=\sum_{d\mid s}d^{-t}\big(
\Sigma_t(\N_{\ell-\Omega(d)}(\sQ_d))-\Sigma_t(\sS_d)\big) \notag \\
&=\sum_{\substack{d\mid s\\d\ne s}}d^{-t}\big(
\Sigma_t(\N_{\ell-\Omega(d)}(\sQ_d))-\Sigma_t(\sS_d)\big),
\end{align}
since $\N_{\ell-\Omega(s)}(\sQ_s)=\sS_s=\{1\}$.

However, $\sQ_d$ is a proper subset of $\sP$ when $d$ is a proper divisor of $s$, and so hypothesis~$(i)$ tells us that $\N_{\ell-\Omega(d)}(\sQ_d)$ is $t$-best among primitive subsets of $\Nge{\ell-\Omega(d)}(\sQ_d)$.
In particular, $\sS_d$ is a primitive subset of $\Nge{\ell-\Omega(d)}(\sQ_d)$, and so $\Sigma_t(\sS_d) \le \Sigma_t(\N_{\ell-\Omega(d)}(\sQ_d))$ for every proper divisor $d$ of~$s$.
Consequently, equation~\eqref{eq:lunch2} demonstrates that $\Sigma_t(\sS) \le \Sigma_t(\N_\ell(\sP))$. Finally, since $\ell \ge k$, hypothesis $(ii)$ tells us that $\Sigma_t(\N_\ell(\sP)) \le \Sigma_t(\N_k(\sP))$;
thus we have derived the required inequality $\Sigma_t(\sS) \le \Sigma_t(\N_k(\sP))$.
\end{proof}

The proof of Proposition~\ref{prop:what} relies upon one remaining statement, which has an elegant proof from the field of algebraic combinatorics. For every natural numbers $k$ and $m$, define the polynomial
\begin{equation} \label{curious function}
h_k(x_1,\dots,x_m) = \sum_{1\le j_1\le \cdots\le j_k\le m}x_{j_1}x_{j_2}\cdots x_{j_k}.
\end{equation}

\begin{lemma}
\label{lem:curious}
Let $(x_1,\dots,x_m)$ be an $m$-tuple of nonnegative real numbers. If $h_1(x_1,\dots,x_m)\ge h_2(x_1,\dots,x_m)$, then $h_k(x_1,\dots,x_m) \ge h_{k+1}(x_1,\dots,x_m)$ for all $k\in\N$.
\end{lemma}

\begin{proof}
For ease of notation, we suppress the dependence of the polynomials $h_k$ on the quantities $x_1,\dots,x_m$. For any $k\in\N$, the Jacobi--Trudi identity tells us that the determinant
$$
\det\begin{pmatrix}h_{k+1}&h_k\\h_{k+2}&h_{k+1}\end{pmatrix} = h_{k+1}^2-h_kh_{k+2}
$$
is equal to the Schur function $s_\lambda$ corresponding to the partition $\lambda=(k+1,k+1)$; see, for example, Macdonald~\cite[Ch.~I, \S 3, Eq.~(4.3)]{McD}. Since the monomials comprising $s_\lambda$ have nonnegative coefficients (and we are evaluating at nonnegative real numbers $x_1,\dots,x_m$), this determinant is nonnegative, which implies that $h_k/h_{k+1}\le h_{k+1}/h_{k+2}$ for each $k\in\N$. However, our assumption is that $1 \le h_1/h_2$, and therefore $1 \le h_k/h_{k+1}$ for all $k\in\N$, as required.
\end{proof}

\begin{proposition}
\label{prop:what}
Let $t>1$ be a real number, and let $\sP$ be a finite set of primes satisfying the inequality $\Sigma_t(\sP) \le 1+\sqrt{1-\Sigma_{2t}(\sP)}$. Then
\begin{itemize}
\item[$(i)$] for every proper subset $\sQ$ of $\sP$ and for every $k\in\N$, the primitive set
$\N_k(\sQ)$ is $t$-best among all primitive subsets of $\Nge{k}(\sQ)$;
\item[$(ii)$] the inequality
$\Sigma_t(\N_k(\sP))\ge\Sigma_t(\N_{k+1}(\sP))$ holds
for all $k\in\N$;
\item[$(iii)$] the primitive set $\N_k(\sP)$ is $t$-best among all primitive subsets of $\Nge{k}(\sP)$, for every $k\in\N$.
\end{itemize}
\end{proposition}

\begin{remark}
Note that the inequality $\Sigma_t(\sP) \le 1+\sqrt{1-\Sigma_{2t}(\sP)}$ is exactly the same as the hypothesis~\eqref{eq:polar-region} of Theorem~\ref{thm:t-best}, while conclusion $(iii)$ is the same as the conclusion of that theorem; hence this proposition implies Theorem~\ref{thm:t-best}. Note also that conclusions $(i)$ and $(ii)$ are the same as the hypotheses of Lemma~\ref{lem:who}, while conclusion $(iii)$ is the same as the conclusion of that lemma. We feel that this redundancy makes clearer the structure of this proposition's proof.
\end{remark}

\begin{proof}
We proceed by induction on the cardinality $\#\sP$. Suppose first that $\#\sP=1$, so that $\sP=\{p\}$ for some prime~$p$. Conclusion $(i)$ holds vacuously. Clearly $\N_k(\sP)=\{p^k\}$ and $\N_{k+1}(\sP)=\{p^{k+1}\}$, from which it follows that
$$
\Sigma_t(\N_k(\sP))=p^{-kt}> p^{-(k+1)t}=\Sigma_t(\N_{k+1}(\sP))
$$
for all $k\in\N$, establishing conclusion~$(ii)$. Finally, since $\sP$ satisfies both $(i)$ and $(ii)$, Lemma~\ref{lem:who} tells us that $\sP$ satisfies conclusion $(iii)$ as well.

Now suppose that $\#\sP>1$. First, let $\sQ$ be any proper subset of~$\sP$. Since $\#\sQ < \#\sP$, the induction hypothesis is that the proposition holds for $\sQ$; in particular, conclusion $(iii)$ holds for~$\sQ$. Thus conclusion $(i)$ holds for~$\sP$.

Turning now to $(ii)$, we begin by treating the case $k=1$, which requires us to establish the inequality $\Sigma_t(\sP) \ge \Sigma_t(\N_2(\sP))$. Note that
\begin{equation} \label{N2big}
\Sigma_t(\sP)^2 = \bigg(\sum_{p\in\sP}p^{-t}\bigg)^2 = 2\sum_{\substack{n\in\N(\sP)\\ \Omega(n)=2}}n^{-t} -\sum_{p\in\sP}p^{-2t}=2\Sigma_t(\N_2(\sP))-\Sigma_{2t}(\sP).
\end{equation}
The inequality $\Sigma_t(\sP) \ge \Sigma_t(\N_2(\sP))$ is therefore equivalent to
$$
\Sigma_t(\sP)^2-2\Sigma_t(\sP)+\Sigma_{2t}(\sP)\le 0,
$$
which holds if and only if
$$
1-\sqrt{1-\Sigma_{2t}(\sP)}\le \Sigma_t(\sP)\le 1+\sqrt{1-\Sigma_{2t}(\sP)}.
$$
The second inequality is exactly the condition we have placed on $\sP$, and so it remains only to prove the first
inequality. First, note that $0 < \Sigma_{2t}(\sP) < \Sigma_2(\sP) \le \Sigma_2(\P) < 1$ by equation~\eqref{less than 1}; therefore $1 - \Sigma_{2t}(\sP) \le \sqrt{1-\Sigma_{2t}(\sP)}$. Consequently,
$$
1 - \sqrt{1-\Sigma_{2t}(\sP)} \le \Sigma_{2t}(\sP) =\sum_{p\in\sP}p^{-2t} < \sum_{p\in\sP}p^{-t}= \Sigma_t(\sP),
$$
as required.

This argument establishes conclusion $(ii)$ in the case $k=1$. However, if we write $\sP = \{p_1,\dots,p_m\}$, note that
$$
\Sigma_t(\N_k(\sP)) = \sum_{n\in\N_k(\sP)} n^{-t} = \sum_{1\le j_1\le \cdots\le j_k\le m} p_{j_1}^{-t}p_{j_2}^{-t}\cdots p_{j_k}^{-t} = h_k\big( p_1^{-t},\dots,p_m^{-t} \big)
$$
using the notation~\eqref{curious function}. We have just shown that $h_1\big( p_1^{-t},\dots,p_m^{-t} \big) \ge h_2\big( p_1^{-t},\dots,p_m^{-t} \big)$, and so Lemma~\ref{lem:curious} imples that $h_k\big( p_1^{-t},\dots,p_m^{-t} \big) \ge h_{k+1}\big( p_1^{-t},\dots,p_m^{-t} \big)$ for all $k\in\N$, which establishes conclusion $(ii)$ in full. Finally, since $\sP$ satisfies both $(i)$ and $(ii)$, Lemma~\ref{lem:who} tells us that $\sP$ satisfies conclusion $(iii)$ as well, which completes the proof of the proposition.
\end{proof}

\bigskip
\noindent{\bf Acknowledgements.} The authors thank
David Speyer for suggesting the proof of Lemma~\ref{lem:curious}.
This work began during visits by the authors to Brigham Young
University, whose hospitality and support are gratefully acknowledged.

\end{document}